\documentclass[12pt]{amsart}
\usepackage{amsmath, amsthm, amssymb}
\usepackage[initials]{amsrefs}
\title{On VC-density in VC-minimal theories}
\author{Vincent Guingona}
\address{Ben-Gurion University of the Negev}
\urladdr{http://www.math.bgu.ac.il/~guingona/}
\email{guingona@math.bgu.ac.il}
\date{\today}

\newtheorem{thm}{Theorem}[section]
\newtheorem{cor}[thm]{Corollary}
\newtheorem{lem}[thm]{Lemma}

\newtheorem{ques}[thm]{Open Question}

\theoremstyle{remark}
\newtheorem{rem}[thm]{Remark}

\theoremstyle{definition}
\newtheorem{defn}[thm]{Definition}

\newcommand{\concat}{{}^{\frown} }

\newcommand{\opp}{\mathrm{opp} }

\newcommand{\nunlhd}{\not\hspace{-5pt}\unlhd\hspace{4pt} }
\newcommand{\VCd}{\mathrm{vc}^* }
\newcommand{\dist}{\mathrm{dist} }
\newcommand{\diff}{\mathrm{diff} }

\newcommand{\acl}{\mathrm{acl} }
\newcommand{\dcl}{\mathrm{dcl} }
\newcommand{\Mon}{\mathcal{U} }
\newcommand{\Td}{\mathcal{T} }
\newcommand{\Fr}{\mathcal{F} }
\newcommand{\eq}{\mathrm{eq} }
\newcommand{\BigO}{\mathcal{O} }
\newcommand{\Vir}{\mathcal{V} }
\newcommand{\Pow}{\mathcal{P} }

\begin{document}

\begin{abstract}
 We show that any formula with two free variables in a VC-minimal theory has VC-codensity at most two.  Modifying the argument slightly, we give a new proof of the fact that, in a VC-minimal theory where $\acl^\eq = \dcl^\eq$, the VC-codensity of a formula is at most the number of free variables (from \cites{adhms, gl}).
\end{abstract}

\maketitle

\section{Introduction}\label{Sec_Intro}

There is a strong connection between the study of NIP theories from model theory and the study of Vapnik-Chervonenkis dimension and density from probability theory.  Indeed, as first noted in \cite{l92}, a theory has NIP if and only if all definable families of sets have finite VC-dimension.  Moreover, a definable family of sets has finite VC-dimension if and only if it has finite VC-density.  Although VC-dimension provides a reasonable measure of the ``complexity'' of a definable set system in an NIP theory, it is highly susceptible to ``local effects.''  Indeed a theory that is relatively ``tame'' globally but locally codes the power set of a large finite set will have high VC-dimension.  On the other hand, VC-density is, in some respect, a much more natural measurement of complexity, impervious to such local complexity.  Moreover, it is closely related to other measurements of complexity in NIP theories, most notably, the dp-rank (see, for example, \cites{adhms, gh, kou}).

In the pair of VC-density papers by M. Aschenbrenner, A. Dolich, D. Haskell, D. Macpherson, and S. Starchenko \cites{adhms, adhms2}, significant progress was made toward understanding VC-density in some NIP theories.  Bounds were given for VC-density in weakly o-minimal theories, strongly minimal theories, the theory of the $p$-adics, the theory of algebraically closed valued fields, and the theory of abelian groups.  However, many questions were left open.  Perhaps the most interesting is the relationship between dp-rank and VC-density.

\begin{ques}\label{Ques_dprankVCden}
 In a theory $T$, is it true that a partial type $\pi(y)$ has dp-rank $\le n$ if and only if every formula $\varphi(x; y)$ has VC-density $\le n$ with respect to $\pi(y)$?
\end{ques}

A simpler question, implied by this and the subadditivity of the dp-rank \cite{kou}, is the following:

\begin{ques}\label{Ques_dpminVCden}
 If $T$ is dp-minimal and $\varphi(x;y)$ is any formula, then does $\varphi$ have VC-density $\le |y|$?
\end{ques}

Both of these appear to be very difficult questions to answer.  So, we can ask an ostensibly easier question, replacing dp-minimality by something stronger.

VC-minimality was first introduced by H. Adler in \cite{adl}.  A theory is VC-minimal if all definable families of sets in one dimension are ``generated'' by a collection of definable sets with VC-codimension $\le 1$.  It turns out that all VC-minimal theories are indeed dp-minimal.  Moreover, due to the close relationship between VC-dimension and VC-density, something can be said about VC-density in VC-minimal theories, to some degree.  However, the primary question on computing VC-density in VC-minimal theories is still open.

\begin{ques}\label{Ques_vcminVCden}
 If $T$ is VC-minimal and $\varphi(x;y)$ is any formula, then does $\varphi$ have VC-density $\le |y|$?
\end{ques}

In this paper, we provide partial solutions to this question.  The primary result is the following, which says this holds when $|y| \le 2$.

\begin{thm}\label{Thm_VCminVCden}
 If $T$ is VC-minimal and $\varphi(x;y)$ is any formula with $|y| \le 2$, then $\varphi$ has VC-density $\le 2$.
\end{thm}

In particular, since the theory of algebraically closed valued fields is VC-minimal, this provides a new result for this theory.

Although this theorem seems quite distant from answering Open Question \ref{Ques_vcminVCden}, the proof is unique, employing a combinatorial method for dealing with directed systems, and may be of independent interest.  For example, we will discuss using the method to provide an entirely new proof for the weakly o-minimal case in \cite{adhms}.

\section{VC-Codensity and Directedness}\label{Sec_VCCoden}

\subsection{VC-codensity}

Fix $T$ a complete first-order theory in a language $L$ with monster model $\Mon$.  If $x$ is a tuple of variables, let $|x|$ denote the length of $x$ and let $\Mon_x$ denote the set $\Mon^{|x|}$ (more generally, if $L$ is multisorted, we let $\Mon_x$ be the elements in $\Mon$ of the same sort as $x$).

If $\Phi(x;y) := \{ \varphi_i(x;y) : i \in I \}$ is a set of formulas and $B \subseteq \Mon_y$, let $S_\Phi(B)$ be the \emph{$\Phi$-type space over $B$}.  That is, $S_\Phi(B)$ is the set of all maximal consistent subsets of
\[
 \{ \varphi_i(x; b)^t : b \in B, i \in I, t < 2 \}.
\]
Here we use the standard notation $\theta(x)^1 = \theta(x)$ and $\theta(x)^0 = \neg \theta(x)$ for formulas $\theta(x)$.  Moreover, if $P$ is an expression that can either be true or false, then we will denote $\theta^P = \theta$ if $P$ is true and $\theta^P = \neg \theta$ if $P$ is false.  For each $p \in S_\Phi(B)$, there exists a unique $s \in {}^{B \times \Phi} 2$ (i.e., $s : B \times \Phi \rightarrow \{ 0, 1 \}$), namely the one such that
\[
 p(x) = \{ \varphi(x; b)^{s(b,\varphi)} : b \in B, \varphi \in \Phi \}.
\]
Hence, $|S_\Phi(B)| \le 2^{|B| \cdot |\Phi|}$.  However, in interesting cases (i.e., when $\Phi$ has NIP), there is a polynomial bound instead of an exponential one.  This leads to the following definition.

\begin{defn}[VC-codensity]\label{Defn_VCcoDen}
 Given a finite set of formulas $\Phi(x; y)$ and a real number $\ell$, we say that $\Phi$ has \emph{VC-codensity $\le \ell$}, denoted $\VCd(\Phi) \le \ell$, if there exists $K < \omega$ such that, for all finite $B \subseteq \Mon_y$,
 \[
  | S_\Phi(B) | \le K \cdot |B|^\ell.
 \]
 If no such number exists, say the VC-codensity is infinite ($\VCd(\Phi) = \infty$).
\end{defn}

A set of formulas $\Phi(x; y)$ has \emph{VC-density} $\le \ell$ if $\Phi^\opp(y;x) := \Phi(x;y)$ has VC-codensity $\le \ell$ (when we exchange the parametrization).  For simplicity, we will only consider VC-codensity in this paper.

Consider the function $\pi_T : \omega \rightarrow \mathbb{R}_{\ge 0} \cup \{ \infty \}$ we call the \emph{VC-codensity function}, namely
\[
 \pi_T(n) = \sup \{ \VCd(\Phi) : \Phi(x;y) \text{ is finite with } |x| = n \}.
\]
Notice that $\pi_T(n) \ge n$.  This is witnessed by the single formula
\[
 \varphi(x_0, ..., x_{n-1}; y) = \bigvee_{i < n} x_i = y.
\]

As usual, if $\Phi = \{ \varphi \}$ is a singleton, then $S_\varphi(B) = S_\Phi(B)$ and $\VCd(\varphi) = \VCd(\Phi)$.  By coding tricks, it suffices to assume that the $\Phi$ in the definition of $\pi_T(n)$ are all singletons.

\begin{lem}[Sauer-Shelah Lemma]\label{Lem_Sauer}
 The following are equivalent for a formula $\varphi(x;y)$.
 \begin{enumerate}
  \item $\varphi$ has NIP,
	\item $\VCd(\varphi)$ is finite.
 \end{enumerate}
\end{lem}

Even if a theory $T$ has NIP, this does not guarantee that $\pi_T(n)$ is finite.  Indeed we may have formulas $\varphi(x;y)$ with $|x|=1$ and have $\VCd(\varphi)$ arbitrarily large, even in the stable context.  For example, countably many independent equivalence relations.

On the other hand, many interesting theories $T$ have at some bound on $\pi_T(n)$.  For example, any weakly o-minimal theory $T$ has $\pi_T(n) = n$ for all $n < \omega$ (Theorem 6.1 of \cite{adhms}).  The theory of the $p$-adics $T$ has $\pi_T(n) \le 2n-1$ (Theorem 1.2 of \cite{adhms}).  The theory of algebraically closed fields $T$ has $\pi_T(n) \le 2n$ (Corollary 6.3 of \cite{adhms}).

The primary problem in the study of VC-codensity for theories is to determine when we can bound $\pi_T(n)$.  Specifically, what conditions on $T$ guarantee that $\pi_T(n) = n$?

\subsection{VC-minimaility}

For a set $X$ and a set system on $X$, $\mathcal{C} \subseteq \Pow(X)$, then we say that $\mathcal{C}$ is \emph{directed} if, for all $A, B \in \mathcal{C}$, one of the following holds:
\begin{itemize}
 \item $A \subseteq B$,
 \item $B \subseteq A$, or
 \item $A \cap B = \emptyset$.
\end{itemize}
Note that, if $\mathcal{C}$ is directed, then $(\mathcal{C}; \supseteq)$ is a forest (and, if $X \in \mathcal{C}$, then it is a tree with root $X$).

In general, we can convert from formulas to set systems.  If $\theta(x)$ is a formula (possibly with parameters), then let $\theta(\Mon) := \{ a \in \Mon_x : \models \theta(a) \}$.  Suppose $\Delta = \{ \delta_i(x; y_i) : i \in I \}$ is a set of partitioned formulas (where $y_i$ is allowed to vary but $x$ is fixed and usually $|x| = 1$) and consider the set system on $\Mon_x$,
\[
 \mathcal{C}_\Delta := \{ \delta_i(\Mon; b) : i \in I, b \in \Mon_{y_i} \}
\]

\begin{defn}[Directedness]\label{Defn_Directedness}
 We say $\Delta$ is \emph{directed} if $\mathcal{C}_\Delta$ is directed.  A formula $\delta(x;y)$ is \emph{directed} if $\{ \delta(x;y) \}$ is directed.  An \emph{instance} of $\Delta$ is a formula of the form $\delta_i(x; b)$ for some $i \in I$ and $b \in \Mon_{y_i}$.
\end{defn}

\begin{defn}[VC-minimality, \cite{adl}]\label{Defn_VCMin}
 A theory $T$ is \emph{VC-minimal} if there exists a directed set of formulas $\Delta$ with $|x| = 1$ such that all formulas (with parameters) $\theta(x)$ are $T$-equivalent to a boolean combination of instances of $\Delta$.
\end{defn}

We will call the family $\Delta$ the \emph{generating family} and we will call instances of $\Delta$ \emph{balls}.  Throughout the remainder of this paper, we assume $T$ is a VC-minimal theory.

\begin{defn}[Unpackable, \cite{cs}]\label{Defn_Unpackable}
 A directed family $\Delta$ is \emph{unpackable} if no instance of $\Delta$ is $T$-equivalent to a disjunction of finitely many proper instances.
\end{defn}

The following is a fundamental decomposition theorem for formulas in VC-minimal theories.

\begin{thm}[Theorem 4.1 of \cite{fg}]\label{Thm_uBalls}
 For all formulas $\varphi(x; y)$ (with $|x| = 1$), there exists a directed formula $\delta(x; z)$, $N < \omega$, and formulas $\psi_i(x; y)$ for $i < 2N$ such that
 \begin{itemize}
  \item for all $b \in \Mon_y$, there exists $n \le N$ and $c_i, ..., c_{n-1} \in \Mon_z$, $\psi_i(x; b)$ is $T$-equivalent to $\bigvee_{i < n} \delta(x; c_i)$, and
	\item for all $b \in \Mon_y$, $\varphi(x; b)$ is $T$-equivalent to
	 \[
	  \bigwedge_{i < N} \psi_{2i}(x; b) \wedge \neg \psi_{2i+1}(x; b).
	 \]
 \end{itemize}
\end{thm}

\begin{rem}[Finite VC-minimality and u-balls]\label{Rem_FinVC}
 Throughout this paper, we will only be working with local properties of a VC-minimal theory (i.e., computing VC-codensity).  In light of Theorem \ref{Thm_uBalls}, we may assume that the generating family $\Delta$ is a singleton, $\{ \delta \}$.  Moreover, we may assume that the formula $\varphi(x; y)$ considered is such that, for all $b \in \Mon_y$, there exists $N < \omega$ such that $\varphi(x; b)$ is $T$-equivalent to a disjunction of at most $N$ instances $\delta$.  In \cite{cs}, these are called u-balls.  
\end{rem}

\begin{rem}[VC-minimality when $\acl = \dcl$]\label{Rem_acldcl}
 If $\acl^\eq = \dcl^\eq$ in $T$, then we may actually assume that all formulas are balls.  For example, suppose $\varphi(x; y)$ is a formula and $N < \omega$ are such that, for all $b \in \Mon_y$, there exists $n \le N$, $c_0, ..., c_{n-1} \in \Mon_z$ such that,
 \[
  \models (\forall x) \left( \varphi(x; b) \leftrightarrow \bigvee_{i < n} \delta(x; c_i) \right).
 \]
 Then, in particular, $c_i / \delta \in \acl^\eq(b)$, hence $c_i / \delta \in \dcl^\eq(b)$.  Thus, there exist formulas $\delta_i(x; y)$ for $i < N$ such that
 \begin{itemize}
  \item $\{ \delta_i(x; y) : i < N \}$ is directed, and
	\item $\varphi(x; y)$ is $T$-equivalent to $\bigvee_{i < n} \delta_i(x; y)$.
 \end{itemize}
 For more details, see Section \ref{Sect_FullyVCMin} below.
\end{rem}

Now for any set $C \subseteq \Mon_z$ and any directed set $\Delta(x;z)$, there is a quasi-forest structure on $C \times \Delta$.  Namely,
\[
 \langle c_0, \delta_0 \rangle \unlhd \langle c_1, \delta_1 \rangle \text{ if } \models (\forall x)(\delta_1(x; c_1) \rightarrow \delta_0(x; c_0)).
\]
This is a quasi-forest instead of a true forest because we could have that $\delta_0(x; c_0)$ and $\delta_1(x; c_1)$ are unequal but $T$-equivalent.  Let $\Fr(C, \Delta) := (C \times \Delta; \unlhd)$ denote this quasi-forest.  We can expand $C \times \Delta$ by a ``root,'' call it $0$, and set $0 \unlhd \langle c, \delta \rangle$ for all $c \in C$ (and $\langle c, \delta \rangle \unlhd 0$ if $\models (\forall x)\delta(x; c)$).  Then, $\Td(C,\Delta) := (C \times \Delta \cup \{ 0 \}; \unlhd)$ is a quasi-tree.

\begin{rem}\label{Rem_Types}
 Suppose $C$ and $\Delta$ are finite.  Each type in the $\Delta$-type space $S_\Delta(C)$ corresponds to a node in the quasi-tree $\Td(C,\Delta)$ (and, if $\Delta$ is an unpackable, this is a bijective correspondence).  To see this, for each $c \in C$ and $\delta \in \Delta$, consider the generic $\Delta$-type corresponding to the (interior) of the ball $\delta(x; c)$, namely
 \[
  \nu_{c,\delta}(x) = \left\{ \delta'(x; c')^{\text{iff } \models (\forall x)(\delta(x; c) \rightarrow \delta'(x; c'))} : c' \in C, \delta' \in \Delta \right\}
 \]
 and let $\nu_0(x) = \{ \neg \delta'(x; c') : c' \in C, \delta' \in \Delta \}$, the generic type of the root (needed if no ball is the whole space yet balls do not cover the whole space).  Define the \emph{virtual type space}
 \[
  \Vir_\Delta(C) = \{ \nu_{c,\delta}(x) : c \in C, \delta \in \Delta \} \cup \{ \nu_0 \}.
 \]
 By the directedness of $\Delta$, it is not hard to see that $S_\Delta(C) \subseteq \Vir_\Delta(C)$.  Note, however, that if $\Delta$ is packable, then this may be a proper inclusion.  If a ball is the union of proper subballs, then the generic type corresponding to this ball is inconsistent.  This is why we refer to these as \emph{virtual} types.

 It is necessary to consider only finite $C$ and $\Delta$.  For example, in the theory of dense linear orders, if $C = \mathbb{Q}$ and $\Delta = \{ x < y \}$, then $S_\Delta(C)$ has size $2^{\aleph_0}$.  On the other hand, as defined, clearly $\Vir_\Delta(C)$ is countable.  Indeed, $\Vir_\Delta(C)$ misses all non-proper cuts.
\end{rem}

In particular we get that, for finite $C$ and finite $\Delta$,
\[
 |S_\Delta(C)| \le |\Delta| \cdot |C| + 1.
\]
Thus, for a VC-minimal theory $T$,
\[
 \pi_T(1) = 1.
\]

This leads to the primary open question regarding VC-minimal theories (and VC-codensity), a restatement of Open Question \ref{Ques_vcminVCden} in this terminology.

\begin{ques}[VC-codensity in VC-minimal theories]\label{Ques_VCMin}
 Is it true that, in all VC-minimal theories $T$, for all $n < \omega$, $\pi_T(n) = n$?
\end{ques}

For example, the theory of algebraically closed valued fields (ACVF) is VC-minimal.  Therefore, answering this question would improve the bound given in \cite{adhms}.  In this paper, we give several partial results to this question.  In particular, we get a new result for ACVF.

In the next subsection, in light of Remark \ref{Rem_Types}, we will be working with quasi-forests $\Fr$, keeping in mind that these quasi-forests will correspond to $\Delta$-types spaces, hence aid us in computing the VC-codensity of formulas in $T$.

\subsection{Quasi-forests}

Let $(F; \unlhd)$ be a finite quasi-forest.  For each $t \in F$, define $\nu(t) = \{ s \in F : s \unlhd t \}$.  In the model theory context, if we think of $t$ as the parameter, then $\nu(t)$ is the generic type corresponding to $t$.  Then we can create the ``tree of types,'' namely
\[
 \Vir(F, \unlhd) = \{ \nu(t) : t \in F \} \cup \{ \emptyset \}
\]
ordered via inclusion (i.e., for $p, q \in \Vir(F, \unlhd)$, $p \unlhd q$ if $p \subseteq q$).  Then, it is easy to see that $\Vir(F, \unlhd) \setminus \{ \emptyset \}$ is isomorphic to the partial order generated by $(F, \unlhd)$ via the map $\nu$.  Moreover, for $p, q \in \Vir(F, \unlhd)$, $(p \cap q) \in \Vir(F, \unlhd)$ is the tree-theoretic meet of $p$ and $q$.

From an arbitrary linear ordering $\le^*$ on each level of $\Vir(F, \unlhd)$, we construct a linear ordering on $\Vir(F, \unlhd)$, $\le$ extending the tree order as follows:
\begin{itemize}
 \item If $p \subset q$, then $p < q$.
 \item If $p$ and $q$ are $\subseteq$-incomparable, let $p^*$ be such that $(p \cap q) \subset p^* \subseteq p$ and $p^*$ minimal such and similarly define $q^*$.  Then $p < q$ if $p^* <^* q^*$.
\end{itemize}
Note that in the case where $(F; \unlhd) = \Fr(C, \Delta)$ for a directed set of formulas $\Delta(x;z)$ and $C \subseteq \Mon_z$ as above, the ordering $\le$ we get here corresponds to the ``convex ordering'' (see \cites{fg2, gl}).  That is, the instances of $\delta$ are convex in this ordering.  Formally:

\begin{lem}\label{Lem_ConvexOrdering}
 For all $t \in F$, the set $\chi(t) := \{ p \in \Vir(F, \unlhd) : t \in p \}$ is $\le$-convex.
\end{lem}

\begin{proof}
 Suppose $t \in F$ and $p < r < q$ with $p, q \in \chi(t)$.  In particular, $t \in (p \cap q)$.  If $(p \cap q) \subseteq r$, then $t \in r$ so $r \in \chi(t)$.  So suppose this fails.  If $r \subset (p \cap q)$, then in particular $r \subset p$, hence $r < p$, a contradiction.  Thus $r$ and $(p \cap q)$ are $\subseteq$-incomparable.  Thus $(r \cap q) = (r \cap p) \subset (p \cap q)$.  Hence, by the second part of the definition of ordering, either $r < p,q$ or $p,q < r$ (depending on $\le^*$).  Contradiction.
\end{proof}

For each $p, q \in \Vir(F, \unlhd)$, define
\[
 \diff(p,q) = (p \triangle q) \text{ and } \dist(p,q) = | \diff(p,q) |.
\]

\begin{lem}\label{Lem_SumDist}
 For any sequence $p_0 < ... < p_m$ from $\Vir(F, \unlhd)$,
 \[
  \sum_{i < m} \dist(p_i, p_{i+1}) \le 2|F|.
 \]
\end{lem}

\begin{proof}
 For any $t \in F$, for all $i < m$, $t \in \diff(p_i, p_{i+1}) = (p_i \triangle p_{i+1})$ if and only if
 \begin{itemize}
  \item $p_i \in \chi(t)$ and $p_{i+1} \notin \chi(t)$, or
	\item $p_i \notin \chi(t)$ and $p_{i+1} \in \chi(t)$.
 \end{itemize}
 By Lemma \ref{Lem_ConvexOrdering}, $\chi(t)$ is $\le$-convex, so, for each $t \in F$, there exists at most two $i < m$ such that $t \in \diff(p_i, p_{i+1})$.  The conclusion follows.
\end{proof}

\subsection{The quasi-forest $\Fr(C, \Delta)$}

Fix $\Delta(x;z)$ a finite directed set, $C \subseteq \Mon_z$ finite, and consider $\Fr(C, \Delta)$ as defined above.  Notice that $\Vir_\Delta(C)$ is isomorphic to $\Vir(\Fr(C, \Delta))$ via $\nu_0 \mapsto \emptyset$ and $\nu_{c,\delta}(x) \mapsto \nu(\langle c, \delta \rangle)$.  Thus, for $p, q \in \Vir_\Delta(C)$, we define
\[
 \diff(p,q) = (p \triangle q) \text{ and } \dist(p,q) = | \diff(p,q) |.
\]
Clearly this corresponds via our isomorphism to the definition above.  By Lemma \ref{Lem_SumDist}, we get the following.

\begin{lem}\label{Lem_SumDist2}
 There exists $\le$ a linear order on $\Vir_\Delta(C)$ such that, for all $p_0 < ... < p_m$ from $\Vir_\Delta(C)$,
 \[
  \sum_{i < m} \dist(p_i, p_{i+1}) \le 2 |C| |\Delta|.
 \]
\end{lem}

This lemma is vital to our method of counting types in VC-minimal theories, as we will demonstrate in the next section using the test case of fully VC-minimal theories.

\section{Test Case: Fully VC-Minimal Theories}\label{Sect_FullyVCMin}

\begin{defn}[Definition 3.9 of \cite{gl}]\label{Defn_FullyVCMin}
 A theory $T$ is \emph{fully VC-minimal} if there exists a directed family of formulas $\Delta$ with $|x| = 1$ such that, for all formulas $\varphi(x;y)$ with $|x| = 1$ and $y$ arbitrary, $\varphi(x;y)$ is $T$-equivalent to a boolean combination of elements of $\Delta$.
\end{defn}

As noted in Remark \ref{Rem_acldcl} above, if $T$ is VC-minimal and $\acl^\eq = \dcl^\eq$, then $T$ is fully VC-minimal.  For example, any weakly o-minimal theory is fully VC-minimal.  On the other hand, ACVF and even ACF are not fully VC-minimal.  See Example 3.15 of \cite{gl} for details.

\begin{thm}[Theorem 3.14 of \cite{gl}]\label{Thm_FullyVCMinLowDen}
 If $T$ is fully VC-minimal, then $\pi_T(n) = n$ for all $n < \omega$.  That is, for all formulas $\varphi(x;y)$, the VC-codensity of $\varphi$ is $\le |x|$.
\end{thm}

The proof presented in \cite{gl} goes through UDTFS-rank, similar to the proof for weakly o-minimal theories given in \cite{adhms}, but in this section, we will sketch an alternate proof using ``pure combinatorics.''  We use this to motivate the process by which we compute the VC-codensity of some formulas in general VC-minimal theories.

We prove Theorem \ref{Thm_FullyVCMinLowDen} by induction on $n$.  If $n = 1$, fix $\varphi(x;y)$ with $|x| = 1$.  Fix a finite directed $\Delta(x;y)$ such that $\varphi(x;y)$ is a boolean combination of elements of $\Delta$.  Then, for any finite $B \subseteq \Mon_y$,
\[
 |S_\varphi(B)| \le |S_{\Delta}(B)|.
\]
However, as argued above, $|S_{\Delta}(B)| \le | \Delta | \cdot |B| + 1$, which is linear in $|B|$.  Hence, $\VCd(\varphi) \le 1$.

In general, fix $n > 1$ and consider $\varphi(x_0, x_1; y)$, where $|x_0| = 1$ and $|x_1| = n-1$.  Repartition $\varphi$ via
\[
 \hat{\varphi}(x_0; x_1, y) = \varphi(x_0, x_1; y)
\]
and, as before, there exists a finite directed $\Delta_0(x_0; x_1, y)$ such that $\hat{\varphi}$ is a boolean combination of elements of $\Delta_0$.  Again, for any finite $B \subseteq \Mon_y$ and any $a_1 \in \Mon_{x_1}$,
\[
 |S_{\hat{\varphi}}(a_1 \concat B)| \le |S_{\Delta_0}(a_1 \concat B)| \le |\Delta_0| \cdot |B| + 1.
\]
But how do we use this to count $\varphi$-types over $B$ instead of $\hat{\varphi}$-types over $a_1 \concat B$?  We describe the quasi-forest structure given by $\Delta_0(x_0; a_1, B)$.

For each $\delta(x_0; x_1, y), \delta'(x_0; x_1, y) \in \Delta_0$, let
\[
 \psi_{\delta, \delta'}(x_1; y, y') := \forall x_0 \left( \delta'(x_0; x_1, y') \rightarrow \delta(x_0; x_1, y) \right).
\]
Notice that, for all $a_1 \in \Mon_{x_1}$, for all $b, b' \in B$, and for all $\delta, \delta' \in \Delta_0$,
\[
 \models \psi_{\delta, \delta'}(a_1; b, b') \text{ iff } \langle a_1, b, \delta \rangle \unlhd \langle a_1, b', \delta' \rangle,
\]
with the quasi-forest structure $\Fr(a_1 \concat B, \Delta_0)$ described in Remark \ref{Rem_Types}.

\begin{lem}[Quasi-forests determined by $\Psi$-types]\label{Lem_TreesDeter}
 If $p(x_1) \in S_{\Psi}(B \times B)$, $a_1, a_1' \models p$, then, as quasi-forests,
 \[
  \Fr(a_1 \concat B, \Delta_0) \cong \Fr(a_1' \concat B, \Delta_0)
 \]
 via the map $\langle a_1, b, \delta \rangle \mapsto \langle a_1', b, \delta \rangle$.
\end{lem}

\begin{proof}
 For all $b, b' \in B$, $\delta, \delta' \in \Delta_0$,
 \[
  \langle a_1, b, \delta \rangle \unlhd \langle a_1, b', \delta' \rangle \text{ iff. } p(x_1) \vdash \psi_{\delta, \delta'}(x_1, b, b').
 \]
 Since the same holds for $a_1'$, we get $\langle a_1, b, \delta \rangle \unlhd \langle a_1, b', \delta' \rangle$ if and only if $\langle a_1', b, \delta \rangle \unlhd \langle a_1', b', \delta' \rangle$.
\end{proof}

In particular, for any such $p$, we can define the quasi-forest
\[
 \Fr(p,B,\Delta_0) = ( B \times \Delta_0; \unlhd_p ),
\]
where, for all $b, b' \in B$, $\delta, \delta' \in \Delta_0$,
\[
 \langle b, \delta \rangle \unlhd_p \langle b', \delta' \rangle \text{ iff. } p(x_1) \vdash \psi_{\delta, \delta'}(x_1; b, b').
\]
In particular, for all $a_1 \models p$,
\[
 \Fr(a_1 \concat B, \Delta_0) \cong \Fr(p, B, \Delta_0)
\]
via the map $\langle a_1, b, \delta \rangle \mapsto \langle b, \delta \rangle$.  Similar to the definition of $\nu_{c,\delta}$ as in Remark \ref{Rem_Types} above, for $\langle b, \delta \rangle \in \Fr(p,B,\Delta_0)$, define
\[
 \nu_{p,b,\delta}(x_0; x_1) := \{ \delta'(x_0; x_1, b')^{\text{iff } \langle b', \delta' \rangle \unlhd_p \langle b, \delta \rangle} : b' \in B, \delta' \in \Delta_0 \}.
\]
That is, $\delta'(x_0; x_1, b') \in \nu_{p,b,\delta}$ if and only if
\[
 p(x_1) \vdash (\forall x_0)( \delta(x_0; x_1, b) \rightarrow \delta'(x_0; x_1, b') ).
\]
To deal with the $0$ node, define
\[
 \nu_0(x_0; x_1) := \{ \neg \delta'(x_0; x_1, b') : b' \in B, \delta' \in \Delta_0 \}.
\]
Moreover, as we did in Remark \ref{Rem_Types}, define the \emph{virtual type space}
\[
 \Vir_{\Delta_0}(p, B) := \{ \nu_{p,b,\delta} : \langle b, \delta \rangle \in \Fr(p, B, \Delta_0) \} \cup \{ \nu_0 \}.
\]
In particular, if $a_1 \models p$, then
\[
 \nu_{b,\delta}(x_0) = \nu_{p,b,\delta}(x_0; a_1)
\]
and
\[
 \Vir_{\Delta_0}(a_1 \concat B)(x_0) = \Vir_{\Delta_0}(p, B)(x_0; a_1).
\]
Therefore, we get the following lemma.

\begin{lem}[$\Delta_0$-types determined by $\Psi$-types]\label{Lem_TypesDeter}
 If $p(x_1) \in S_{\Psi}(B \times B)$ and $a_1 \models p$, then
 \[
  S_{\Delta_0}(a_1 \concat B) \subseteq \Vir_{\Delta_0}(p, B)(x_0; a_1).
 \]
\end{lem}

In particular,
\[
 S_\varphi(B) \subseteq \bigcup_{p \in S_\Psi(B \times B)} \Vir_{\Delta_0}(p, B).
\]
Hence, without any more work, we get the bound
\[
 | S_\varphi(B) | \le ( | \Delta_0 | \cdot |B| + 1 ) | S_{\Psi}(B \times B) |.
\]
With no further analysis, induction would yield $\pi_T(n) \le 2^n-1$.  For simplicity, assume $n=2$ and hence $|x_1| = 1$.  Now, by full VC-minimality, there exists a finite directed $\Delta_1(x_1; y, y')$ such that each $\psi(x_1; y, y') \in \Psi$ is a boolean combination of elements of $\Delta_1$.  Therefore,
\begin{align*}
 | S_\varphi(B) | \le \ & ( | \Delta_0 | \cdot |B| + 1 ) | S_{\Delta_1}(B \times B) | \le \\ & ( | \Delta_0 | \cdot |B| + 1 ) \cdot ( | \Delta_1 | \cdot |B|^2 + 1) = \BigO(|B|^3).
\end{align*}
In other words, $\pi_T(2) \le 3$.  We can get $\pi_T(2) = 2$ by paying closer attention to our counting.

Apply Lemma \ref{Lem_SumDist2} to $B \times B$ and $\Delta_1$.  Let $p_0 < ... < p_m$ enumerate $S_{\Delta_1}(B \times B)$ inside $\Vir_{\Delta_1}(B \times B)$, hence
\[
 \sum_{i < m} \dist(p_i, p_{i+1}) \le 2 |B|^2 |\Delta_1|.
\]
For each $\delta_1(x_1; b, b') \in \diff(p_i, p_{i+1})$, at most the inclusion of one formula $\psi_{\delta_0, \delta_0'}(x_1; b, b')$ is changed between $p_i$ and $p_{i+1}$ for some $\delta_0, \delta_0' \in \Delta_0$.  That is, either
\begin{itemize}
 \item $\langle b, \delta_0 \rangle \unlhd_{p_i} \langle b', \delta'_0 \rangle$ and $\langle b, \delta_0 \rangle \nunlhd_{p_{i+1}} \langle b', \delta'_0 \rangle$, or
 \item $\langle b, \delta_0 \rangle \nunlhd_{p_i} \langle b', \delta'_0 \rangle$ and $\langle b, \delta_0 \rangle \unlhd_{p_{i+1}} \langle b', \delta'_0 \rangle$.
\end{itemize}
This results in at most one new virtual $\Delta_0$-type in the corresponding virtual $\Delta_0$-type space $\Vir_{\Delta_0}(p_{i+1}, B)$, namely, the one corresponding to $\langle b', \delta'_0 \rangle$ (whether or not it includes $\delta_0(x_0; x_1, b)$).  Therefore,
\[
 | \Vir_{\Delta_0}(p_{i+1}, B) \setminus \Vir_{\Delta_0}(p_i, B) | \le | \diff(p_i, p_{i+1}) | = \dist(p_i, p_{i+1}).
\]
Therefore,
\begin{align*}
 \left| \bigcup_{p \in S_{\Delta_1}} \Vir_{\Delta_0}(p, B) \right| \le \ & |\Vir_{\Delta_0}(p_0, B)| + \sum_{i < m} | \Vir_{\Delta_0}(p_{i+1}, B) \setminus \Vir_{\Delta_0}(p_i, B) | \le \\ & |\Vir_{\Delta_0}(p_0, B)| + \sum_{i < m} \dist(p_i, p_{i+1}) \le \\ & |\Vir_{\Delta_0}(p_0, B)| + 2 |B|^2 |\Delta_1| \le \\ & 2 |B|^2 |\Delta_1| + |B| |\Delta_0| + 1.
\end{align*}
In particular,
\[
 | S_\varphi(B) | = \BigO( |B|^2 ).
\]
Therefore, $\pi_T(2) = 2$.  The argument is similar for $n > 2$.

\section{General VC-Minimal Theories}

In the general case, by Theorem \ref{Thm_uBalls}, we can assume that the formula whose VC-codensity we are computing is such that each instance is a union of a uniformly bounded number of balls.  However, the problem comes in distinguishing these balls from one another since, in general, they are not individually definable over the parameter used in the instance considered.  So we will need some way of determining irreducible unions of balls and do this relative to a given type-space.

For the remainder of this section, we will give a proof of Theorem \ref{Thm_VCminVCden}, following the outline sketched in Section \ref{Sect_FullyVCMin}.  That is, if $T$ is a VC-minimal theory, then we will show that
\[
 \pi_T(2) = 2.
\]

\subsection{Construction Setup}

Fix $T$ a VC-minimal theory and $\varphi(x;y)$ is a paritioned formula with $|x| = 2$.  Repartition as
\[
 \hat{\varphi}(x_0; x_1, y) := \varphi(x_0, x_1; y).
\]
By Theorem \ref{Thm_uBalls}, there exists $N_0 < \omega$, $\delta_0(x_0;z)$ directed, and $\Gamma_0(x_0; x_1, y)$ a finite set of formulas such that
\begin{itemize}
 \item $\hat{\varphi}(x_0; x_1, y)$ is $T$-equivalent to a boolean combination of elements of $\Gamma_0$, and
 \item each instance of a formula from $\Gamma_0$ is $T$-equivalent to a union of at most $N_0$ instances of $\delta_0$.
\end{itemize}

Fix a finite $B \subseteq \Mon_y$ and we aim to count the size of $S_\varphi(B)$.  As each type in $S_\varphi(B)$ is implied by a type in $S_{\Gamma_0}(B)$, we have
\[
 | S_\varphi(B) | \le | S_{\Gamma_0}(B) |,
\]
so we will count $\Gamma_0$-types over $B$ instead (correctly repartitioned).  Each $\Gamma_0$-type is, in fact, determined by an instance of $\delta_0$, just not necessarily definably over $B$.  For each $a_1 \in \Mon_{x_1}$ and each $c \in \Mon_z$, define
\[
 \nu_{a_1,c}(x_0; x_1) := \left\{ \gamma(x_0; x_1, b)^{\text{iff } \models \forall x_0 (\delta_0(x_0; c) \rightarrow \gamma(x_0; a_1, b) )} : b \in B, \gamma \in \Gamma_0 \right\},
\]
and let
\[
 \Vir(a_1,B) := \{ \nu_{a_1,c}(x_0; x_1) : c \in \Mon_z \} \cup \{ \nu_0(x_0; x_1) \},
\]
where, as before,
\[
 \nu_0(x_0; x_1) := \{ \neg \gamma(x_0; x_1, b) : b \in B, \gamma \in \Gamma_0 \}.
\]
Then, it is easy to check that,
\begin{equation}\label{Eq_VirtualTypes}
 S_{\Gamma_0}(B) \subseteq \bigcup \{ \Vir(a_1; B) : a_1 \in \Mon_{x_1} \}.
\end{equation}
Therefore, it suffices to bound this set.  As we did in Section \ref{Sect_FullyVCMin} above, we will use types in the $x_1$ variable to bound this.  We code this now.

First of all, each formula $\gamma \in \Gamma_0$, each $a_1 \in \Mon_{x_1}$, and each $b \in B$, $\gamma(x_0; a_1, b)$ is $T$-equivalent to a union of at most $N_0$ instances of $\delta_0$.  If $\gamma(x_0, a_1, b)$ is $T$-equivalent to $\bigvee_{i < n} \delta_0(x_0; c_i)$ for $c_i \in \Mon_z$ with $n \le N_0$ minimal such, then we will call the $\delta_0(x_0; c_i)$'s \emph{components} of $\gamma(x_0; a_1, b)$.  Note that, by directedness and minimality of $n$, components are unique up to permutation and $T$-equivalence.

We code the minimal $n \le N_0$ as follows: For each such $n$ and $\gamma \in \Gamma_0$, let
\[
 \psi''_{\gamma,n}(x_1, y, z_0, ..., z_{n-1}) := \forall x_0 \left( \gamma(x_0; x_1, y) \leftrightarrow \bigvee_{i < n} \delta_0(x_0; z_i) \right)
\]
and let
\[
 \psi'_{\gamma,n}(x_1, y) := (\exists z_i)_{i < n} [\psi''_{\gamma,n}] \wedge \neg (\exists z_i)_{i \le n} [\psi''_{\gamma,n+1}].
\]

The next step is to code the $\Gamma_0$-types that correspond to the generic type of a component of $\gamma(x_0; a_1, b)$.  So, for each $m < \omega$, $n \le N_0$, and $\mu \subseteq {}^{m \times \Gamma_0} 2$, consider the following formula, which will code the generic $\Gamma_0$-type over $x_1 \concat \{ w_0, ..., w_{m-1} \}$ generated by components of $\gamma(x_0; x_1, y)$ in terms of $\mu$:
\begin{align*}
 & \psi_{\gamma,n,m,\mu} (x_1, y, w_0, ..., w_{m-1}) := \psi'_{\gamma,n}(x_1, y) \wedge \\ & (\exists z_i)_{i < n} \Bigr[ \psi''_{\gamma,n}(x_1, y, z_0, ..., z_{n-1}) \wedge \\ & \bigwedge_{s \in {}^{m \times \Gamma_0} 2} \left( \bigvee_{i < n} \bigwedge_{j < m, \gamma' \in \Gamma_0} [ \forall x_0 ( \delta_0(x_0; z_i) \rightarrow \gamma'(x_0; x_1, w_j) ) ]^{s(j,\gamma')} \right)^{\text{iff } s \in \mu} \Bigr].
\end{align*}
In other words: Fix $a_1 \in \Mon_{x_1}$, $b \in B$, $b_0, ..., b_{m-1} \in B$ and, any of $c_0, ..., c_{n-1} \in \Mon_z$ such that $\{ \gamma(x; c_i) : i < n \}$ is the set of components of $\gamma(x_0; a_1, b)$.  For each $i < n$, let $\nu(i) \in {}^{m \times \Gamma_0} 2$ be given as follows:
\[
 [\nu(i)](j,\gamma') = 1 \text{ iff. } \models \forall x_0 ( \delta_0(x_0; c_i) \rightarrow \gamma'(x_0, a_1, b_j) ),
\]
which codes the generic $\Gamma_0$-type over $a_1 \concat \{ b_j : j < m \}$ generated by $\delta_0(x_0; c_i)$.  Finally, let
\[
 \mu = \{ \nu(i) : i < n \}.
\]
Then, we see that
\[
 \models \psi_{\gamma,n,m,\mu} (a_1, b, b_0, ..., b_{m-1}).
\]
Moreover, both $n \le N_0$ and $\mu \subseteq {}^{m \times \Gamma_0} 2$ are unique such.  For each $m < \omega$, define
\[
 \Psi_m(x_1; y, w_0, ..., w_{m-1}) := \{ \psi_{\gamma,n,m,\mu} : \gamma \in \Gamma_0, n \le N_0, \mu \in {}^{m \times \Gamma_0} 2 \}.
\]
and let $\Psi := \Psi_{2^{N_0}}$ (our choice to consider only $m \le 2^{N_0}$ will be made clear shortly).  By VC-minimality, there exists $\delta_1(x_1; u)$ directed and $N_1 < \omega$ such that every instance of $\Psi$ is boolean combination of at most $N_1$ instances of $\delta_1$.  Note that this also covers instances of $\Psi_m$ for $m < 2^{N_0}$ by repeating entries.

The goal now is to build a $\Psi$-type space over a set of size $\BigO( |B|^2 )$ such that each type determines a virtual $\Gamma_0$-type space, $\Vir(a_1,B)$.  Then, as we did in Section \ref{Sect_FullyVCMin}, we use this to bound the size of $\Gamma_0$-types over $B$.  To do this, we will need some definitions about how to relate various $\Gamma_0$-type spaces via $\Psi$.

\begin{defn}\label{Defn_Decides}
 Fix $b \in B$, $\gamma \in \Gamma_0$, $m < \omega$, $\overline{c} \in B^m$, and $p(x_1)$ any partial type.  We say that $p$ \emph{decides generic $\Gamma_0$-types of $\gamma(x_0; x_1, b)$ over $\overline{c}$} if, for some $n < \omega$ and $\mu \subseteq {}^{m \times \Gamma_0} 2$,
 \[
  p(x_1) \vdash \psi_{\gamma,n,m,\mu}(x_1, b, \overline{c}).
 \]
 In this case, let $\mu_{p,\gamma,b,\overline{c}} := \mu$ for the unique such $\mu$, the code for said type space.  Also, let $N_{p,\gamma,b,\overline{c}} := | \mu |$, which denotes the size of the type space.
\end{defn}

Clearly if $p(x_1)$ implies a type in $S_{\Psi_m}( \{ \langle b, \overline{c} \rangle \} )$, then $p$ decides generic $\Gamma_0$-types of $\gamma(x_0; x_1, b)$ over $\overline{c}$.  Moreover, if $\overline{c}_0 \subseteq \overline{c}$ is any subsequence, then $p$ decides generic $\Gamma_0$-types of $\gamma(x_0; x_1, b)$ over $\overline{c}_0$ as well.

Note that $N_{p,\gamma,b,\overline{c}} \le N_0$ for any choice of $b$, $\gamma$, $\overline{c}$, and $p$ that decides generic $\Gamma_0$-types of $\gamma(x_0; x_1, b)$ over $\overline{c}$.  This is because, for the formula $\psi_{\gamma,n,|\overline{c}|,\mu}(x_1; b, \overline{c})$ to even be consistent, we must have $N_{p,\gamma,b,\overline{c}} = |\mu| \le n \le N_0$.

Fix $b$, $\gamma$, $\overline{c}$, and $p$ such that $p$ decides generic $\Gamma_0$-types of $\gamma(x_0; x_1, b)$ over $\overline{c}$ and fix $\overline{c}_0 \subseteq \overline{c}$.  Then there is a canonical map $\pi_{p,\gamma,b,\overline{c},\overline{c}_0} : \mu_{p,\gamma,b,\overline{c}} \rightarrow \mu_{p,\gamma,b,\overline{c}_0}$ which is simply projection.  That is, if $\overline{c} = \langle c_i : i < m \rangle$ and $\overline{c}_0 = \langle c_{i_0}, ..., c_{i_{k-1}} \rangle$ for $i_0 < ... < i_{k-1} < n$ and $s \in \mu_{p,\gamma,b,\overline{c}}$, then
\[
 [\pi_{p,\gamma,b,\overline{c},\overline{c}_0}(s)](j,\gamma') = s(i_j,\gamma').
\]
In general, this map is surjective and, when $N_{p,\gamma,b,\overline{c}} = N_{p,\gamma,b,\overline{c}_0}$, it is bijective.  This leads us to the following definition.

\begin{defn}\label{Defn_PreservesSize}
 Fix $b \in B$, $\gamma \in \Gamma_0$, $m, m' < \omega$, $\overline{c} \in B^m$, and $\overline{c}' \in B^{m'}$, and $p(x_1)$ any partial type.  If $p$ decides generic $\Gamma_0$-types of $\gamma(x_0; x_1, b)$ over $\overline{c} \concat \overline{c}'$, then we say that $\overline{c}$ and $\overline{c}'$ \emph{generate the same irreducibles of $\gamma(x_0; x_1, b)$ with respect to $p$} if
 \[
  N_{p,\gamma,b,\overline{c}} = N_{p,\gamma,b,\overline{c} \concat \overline{c}'} = N_{p,\gamma,b,\overline{c}'}.
 \]
\end{defn}

If $\overline{c}$ and $\overline{c}'$ generate the same irreduciles of $\gamma(x_0; x_1, b)$ with respect to $p$, then there is a canonical bijection between $\mu_{p,\gamma,b,\overline{c}}$ and $\mu_{p,\gamma,b,\overline{c}'}$, namely
\[
 \rho_{p,\gamma,b,\overline{c},\overline{c}'} := \pi_{p,\gamma,b,\overline{c} \concat \overline{c}', \overline{c}}^{-1} \circ \pi_{p,\gamma,b,\overline{c} \concat \overline{c}', \overline{c}'}.
\]
Thus, we can use the information of $\Psi$-types over small parts of $B$ and ``glue'' this information together via these bijections to get information about the $\Gamma_0$-type over all of $B$.  We detail this construction now.

\subsection{Primary Construction}

We are now ready to begin the primary construction.  Fix $b \in B$, $\gamma \in \Gamma_0$, and put $<$ an arbitrary linear order on $B$.  Let $\beta_0 := B$, and let $S_0 := S_{\Psi_1}(b \concat \beta_0)$.  Note that, for each $p \in S_0$, $p$ decides generic $\Gamma_0$-types of $\gamma(x_0; x_1, b)$ over $b'$ for all $b' \in B$.  Therefore, for each $p \in S_0$, the following subset is well-defined
\[
 \beta^p_0 := \{ b' \in \beta_0 : N_{p,\gamma,b,b'} > 1 \}.
\]
These correspond to the elements $b'$ that, according to $p$, have more than one $\Gamma_0$-type over $b'$ generic to some component of $\gamma(x_0; x_1, b)$.  This inherits the suborder $<$ from $\beta_0 = B$.  Let
\[
 \beta_{p,1} := \{ \langle b', b'' \rangle : b', b'' \in \beta^p_0, b' < b'' \text{ are $<$-consecutive} \},
\]
which inherits a linear order $<$ from $\beta^p_0$.  Let
\[
 \beta_1 := \left( \bigcup \{ \beta_{p,1} : p \in S_0 \} \right) \cup \{ \langle b', b' \rangle : b' \in \beta_0 \}.
\]
Finally, let
\[
 S_1 := S_{\Psi_2}(b \concat \beta_1).
\]
By including a copy of the diagonal of $\beta_0$ in $\beta_1$, we ensure that, for each $q \in S_1$, there exists $p \in S_0$ such that $q(x_1) \vdash p(x_1)$.

In general, suppose that we have $q \in S_n$ for $n \ge 1$.  By induction, there exists $p \in S_{n-1}$ such that $q \vdash p$.  Supposing $\beta_{p,n}$ is defined with a linear order $<$, let
\[
 \beta^q_{p,n} := \{ \overline{c} \in \beta_{p,n} : N_{q,\gamma,b,\overline{c}} > N_{q,\gamma,b,\overline{c}_0} \text{ or } N_{q,\gamma,b,\overline{c}} > N_{q,\gamma,b,\overline{c}_1} \},
\]
where $\overline{c}_0$ is the first half of $\overline{c}$ and $\overline{c}_1$ is the second half of $\overline{c}$.  That is, $\overline{c}_0$ and $\overline{c}_1$ do not generate the same irreducibles of $\gamma(x_0; x_1, b)$ with respect to $q$.  This is given the suborder $<$ from $\beta_{p,n}$.  Define
\[
 \beta_{q,n+1} := \{ \overline{c} \concat \overline{c}' : \overline{c}, \overline{c}' \in \beta^q_{p,n}, \overline{c} < \overline{c}' \text{ are $<$-consecutive} \},
\]
which inherits the obvious order $<$ from $\beta^q_{p,n}$.  Let
\[
 \beta_{n+1} := \left( \bigcup \{ \beta_{q,n+1} : q \in S_n \} \right) \cup \{ \overline{c} \concat \overline{c} : \overline{c} \in \beta_n \}
\]
and let
\[
 S_{n+1} := S_{\Psi_{2^{n+1}}}(b \concat \beta_{n+1}).
\]
Notice that, as we included a copy of the diagonal of $\beta_n$ in $\beta_{n+1}$, for each $q \in S_{n+1}$, there exists $p \in S_n$ such that $q \vdash p$.

Now, as this all depends on $b \in B$ and $\gamma \in \Gamma_0$, define $\beta_{\gamma,b} := \beta_{N_0}$, let
\[
 \beta := \bigcup \{ b \concat \beta_{\gamma,b} : \gamma \in \Gamma_0, b \in B \},
\]
and let $S := S_{\Psi}(\beta)$.

This concludes our construction.  We need only show that this works.

\subsection{Verifying Construction Works}

First, its clear that $S$ is the set of $\Psi$-types over $\beta \subseteq B^{2^{N_0}+1}$.  We thus need to check that $|\beta| = \BigO(|B|^2)$ and, for each $p(x_1) \in S$, for all $a_1, a_1' \models p$, $\Vir(a_1, B) = \Vir(a_1', B)$.  Moreover, we then need to check that this implies that $|S_\varphi(B)| = \BigO(|B|^2)$, as desired.

For the next two lemmas, fix $b \in B$ and $\gamma \in \Gamma_0$.

\begin{lem}\label{Lem_ConstrEnds}
 The construction terminates by stage $n = N_0$.  That is, for all $q \in S_{N_0}$, $\beta^q_{p,N_0} = \emptyset$ (where $p \in S_{N_0-1}$ is such that $q \vdash p$).
\end{lem}

\begin{proof}
 If any $q$ decides generic $\Gamma_0$-types of $\gamma(x_0; x_1, b)$ over $\overline{c}$, then $N_{q,\gamma,b,\overline{c}} \le N_0$.  Since each iteration of the construction increases the value of this $N_{q,\gamma,b,\overline{c}}$ by at least one, it cannot continue past $N_0$ steps.
\end{proof}

\begin{lem}\label{Lem_ConstrWorks}
 For all $p \in S_{N_0}$, $p$ decides generic $\Gamma_0$-types of $\gamma(x_0; x_1, b)$ over $B$.
\end{lem}

That is, we must show that, for each $p \in S_{N_0}$, there exists $n < \omega$ and $\mu \subseteq {}^{B \times \Gamma_0} 2$ so that
\[
 p(x_1) \vdash \psi_{\gamma,n,|B|,\mu}(x_1, b, B).
\]
We do this by tracing through $B$ using the formulas in $p$.

\begin{proof}
 Fix $n \le N_0$ and $p_0, ..., p_n = p$ with $p_i \in S_i$ and $p_{i+1} \vdash p_i$.  For simplicity of notation, let $\beta^*_i = \beta^{p_i}_{p_{i-1},i}$ for $i \le n$.  Furthermore, choose $n$ such that $\beta^*_n = \emptyset$ and $\beta^*_{n-1} \neq \emptyset$.  By Lemma \ref{Lem_ConstrEnds}, such an $n \le N_0$ exists.  Choose $\overline{c}^* \in \beta^*_{n-1}$ to be $<$-minimal.  Clearly $p$ decides generic $\Gamma_0$-types of $\gamma(x_0; x_1, b)$ over $\overline{c}^*$, so let $\mu^* = \mu_{p,\gamma,b,\overline{c}^*}$.  From this we must build the desired $\mu \subseteq {}^{B \times \Gamma_0} 2$.  Choose $s \in \mu^*$ (which codes a generic $\Gamma_0$-type over $\overline{c}^*$ of some component of $\gamma(x_0; x_1, b)$) and we show how to extend $s$ canonically to a function $s_* : B \times \Gamma_0 \rightarrow 2$.  Setting $\mu = \{ s_* : s \in \mu^* \}$, we are done.

 So fix $b' \in B$, $\gamma' \in \Gamma_0$.  If $b' \in \overline{c}^*$, then $s_*(b',\gamma') = s(i,\gamma')$, where $b'$ is the $i$th element of $\overline{c}^*$.  If $b' \notin \beta^{p_0}_0$, then $N_{p,\gamma,b,b'} = 1$, hence there is a unique $s' \in \mu_{p,\gamma,b,b'}$.  Let $s_*(b',\gamma') = s'(0,\gamma')$.  Otherwise, choose $0 < m < n$ maximal such that $b' \in \overline{c} \in \beta_m$.  Choose $\overline{c}' \in \beta^*_m$ such that:
 \begin{itemize}
  \item if $m = n-1$, $\overline{c}' = \overline{c}^*$, and
	\item if $m < n-1$, $\overline{c}' \subseteq \overline{c}'' \in \beta^*_{m+1}$ and $\overline{c}'$ is $<$-closest such to $\overline{c}$.
 \end{itemize}
 Now, there exists a chain
 \[
  \overline{c}_0 < ... < \overline{c}_k
 \]
 be of $<$-consecutive elements in $\beta^*_m$ with $\overline{c}_0 = \overline{c}$ or $\overline{c}_0 = \overline{c}'$ and similarly for $\overline{c}_k$.  By the choice of $\overline{c}'$, $\overline{c}_i \concat \overline{c}_{i+1} \in \beta_{p_m,m+1}$ yet $\overline{c}_i \concat \overline{c}_{i+1} \notin \beta^*_{m+1}$.  Thereby, $N_{p,\gamma,b,\overline{c}_i} = N_{p,\gamma,b,\overline{c}_i \concat \overline{c}_{i+1}} = N_{p,\gamma,b,\overline{c}_{i+1}}$.  That is, $\overline{c}_i$ and $\overline{c}_{i+1}$ generate the same irreducibles of $\gamma(x_0; x_1, b)$ with respect to $p$.  Hence, there is a bijection between $\mu_{p,\gamma,b,\overline{c}}$ and $\mu_{p,\gamma,b,\overline{c}'}$, namely
 \[
  \rho_{p,\gamma,b,\overline{c}_0,\overline{c}_1} \circ ... \circ \rho_{p,\gamma,b,\overline{c}_{k-1},\overline{c}_k}.
 \]
 Moreover, if $m < n-1$, there is a surjection of $\mu_{p,\gamma,b,\overline{c}''}$ onto $\mu_{p,\gamma,b,\overline{c}}$ by composing the bijection with $\pi_{p,\gamma,b,\overline{c}'',\overline{c}}$.  By induction, this gives us a surjection from $\mu^*$ onto $\mu_{p,\gamma,b,\overline{c}}$.  Let $s'$ be the image of $s$ under this surjection and set $s_*(b',\gamma') = s'(i,\gamma')$, where $b'$ is the $i$th element of $\overline{c}$.  This concludes the construction of $s_*$.

 It is straightforward to check that this works.  Thus, $p$ decides generic $\Gamma_0$-types of $\gamma(x_0; x_1, b)$ over $B$.
\end{proof}

We immediately get the following corollary.

\begin{cor}\label{Cor_ConstrWorks}
 For all $p(x_1) \in S$, for all $a_1, a_1' \models p$, $\Vir(a_1, B) = \Vir(a_1', B)$.
\end{cor}

\begin{proof}
 By Lemma \ref{Lem_ConstrWorks}, for each $b \in B$ and $\gamma \in \Gamma_0$, there exists $n < \omega$ and $\mu_{b,\gamma} \subseteq {}^{B \times \Gamma_0} 2$ such that
\[
 p(x_1) \vdash \psi_{\gamma,n,|B|,\mu_{b,\gamma}}(x_1, b, B).
\]
Therefore, $\models \psi_{\gamma,n,|B|,\mu_{b,\gamma}}(a_1,b,B)$.  Unraveling the definition, we obtain
 \begin{align*}
  \Vir & (a_1, B) = \{ \{ \neg \gamma'(x_0; x_1, b') : b' \in B, \gamma' \in \Gamma_0 \} \} \cup \\ & \{ \{ \gamma'(x_0; x_1, b')^{s(b',\gamma')} : b' \in B, \gamma' \in \Gamma_0 \} : b \in B, \gamma \in \Gamma_0, s \in \mu_{b,\gamma} \}.
 \end{align*}
 Since this also holds for $a_1'$, we get $\Vir(a_1, B) = \Vir(a_1', B)$.
\end{proof}

In light of this, for any $p \in S$, define $\Vir(p, B) = \Vir(a_1, B)$ for any (equivalently all) $a_1 \models p$.

Immediately from this corollary and \eqref{Eq_VirtualTypes} we get that
\[
 |S_\varphi(B)| \le |S_{\Gamma_0}(B)| \le (N_0 \cdot |B| \cdot |\Gamma_0| + 1) \cdot |S|.
\]
By VC-minimality, we know that $|S| \le N_1 |\beta| + 1$, thus we obtain
\[
 |S_\varphi(B)| = \BigO( |B| \cdot |\beta| ).
\]
\emph{A priori} there is no good bound on the size of $\beta$, so this is not immediately helpful.  However, using the trick we employed in Section \ref{Sect_FullyVCMin}, namely Lemma \ref{Lem_SumDist2}, we will obtain a bound on the order of $|B|^2$.

Define constants $K_n$ recursively as follows: Let $K_0 = 1$ and, if $K_n$ is given, let $K_{n+1}$ be such that
\[
 K_{n+1} := 2 K_n ( 1 + 3^n N_1 | \Psi | ).
\]
Finally, let
\[
 K := 3 N_0 N_1 K_{N_0} \cdot | \Psi | \cdot |\Gamma_0|.
\]
Note that these are all independent of $B$.

For the next two lemmas, fix $b \in B$ and $\gamma \in \Gamma_0$.  For the next lemma, consider $n \le N_0$.  For each $\psi \in \Psi_{2^n}$ and $\overline{c} \in \beta_n$, there exists $D_{\psi,\overline{c}} \subseteq \Mon_u$ with $|D_{\psi,\overline{c}}| \le N_1$ so that $\psi(x_1; b, \overline{c})$ is a boolean combination of $\delta_1(x_1; d)$ for $d \in D_{\psi,\overline{c}}$.  Let $D := \bigcup \{ D_{\psi, \overline{c}} : \psi \in \Psi_{2^n}, \overline{c} \in \beta_n \}$.  Hence, for any $q \in S_n$, there exists $p \in S_{\delta_1}(D)$ such that $p(x_1) \vdash q(x_1)$.  Then, we get the following:

\begin{lem}\label{Lem_ShortDist}
 For all $q_0, q_1 \in S_n$ and $p_0, p_1 \in S_{\delta_1}(D)$ with $p_0 \vdash q_0$ and $p_1 \vdash q_1$, we have
 \[
  | \beta_{q_1,n+1} \setminus \beta_{q_0,n+1} | \le 3^n \cdot \dist(p_0, p_1).
 \]
\end{lem}

\begin{proof}
 For $t < 2$, let $q_t = q_{t,n}, ..., q_{t,0}$ be such that $q_{t,i} \in S_i$ and $q_{t,i+1} \vdash q_{t,i}$.  Consider $\psi \in \Psi_{2^n}$ and $\overline{c} \in \beta_n$.  Then $p_0$ and $p_1$ both imply either $\pm \psi(x_1; b, \overline{c})$ and they can disagree only if $\pm \delta(x_1; d) \in \diff(p_0, p_1)$ for some  $d \in D_{\psi,\overline{c}}$.  If they do disagree on $\psi(x_1; b, \overline{c})$, then, for some $i \le n$, this changes at most one element in $\beta^{q_{t,i}}_{q_{t,i-1},i}$ between $t=0$ and $t=1$.  Such a change results in a change of at most $3$ elements in $\beta_{q_{t,i},i+1}$ from $t=0$ and $t=1$ (if say $\overline{c}_0 < \overline{c} < \overline{c}_1$ are consecutive in $\beta^{q_{0,i}}_{q_{0,i-1},i}$ and $\overline{c} \notin \beta^{q_{1,i}}_{q_{1,i-1},i}$, then $\overline{c}_0 \concat \overline{c}_1 \in \beta_{q_{1,i},i+1}$ whereas $\overline{c}_0 \concat \overline{c}, \overline{c} \concat \overline{c}_1 \in \beta_{q_{0,i},i+1}$).  By induction, this causes at most $3^n$ changes in $\beta_{q_t,n+1}$ from $t=0$ and $t=1$.  The conclusion follows.
\end{proof}

With this, we can employ Lemma \ref{Lem_SumDist2} to bound the size of $\beta_n$.

\begin{lem}\label{Lem_SSmall}
 For all $n \le N_0$, $| \beta_n | \le K_n \cdot |B|$.
\end{lem}

\begin{proof}
 We prove this by induction on $n$.  For $n=0$, $\beta_0 = B$, hence $| \beta_0 | = |B|$, as desired.

 Construct $D$ for $\beta_n$ as above.  In particular,
 \[
  |D| \le N_1 \cdot |\Psi| \cdot K_n \cdot |B|
 \]
 and, for each type $q \in S_n$, there exists $p \in S_{\delta_1}(D)$ such that $p \vdash q$.  Let $q_0, q_1, ..., q_{m-1} \in S_n$ be an enumeration of $S_n$ and, for each $t < m$, choose $p_t \in S_{\delta_1}(D)$ such that $p_t \vdash q_t$.  By Lemma \ref{Lem_ShortDist},
 \[
  | \beta_{q_{t+1},n+1} \setminus \beta_{q_t,n+1} | \le 3^n \cdot \dist(p_t, p_{t+1}).
 \]
 Moreover, by definition of $\beta_{n+1}$, we get
 \[
  \beta_{n+1} = \beta_{q_0,n+1} \cup \bigcup_{t < m-1} (\beta_{q_{t+1},n+1} \setminus \beta_{q_{t},n+1}) \cup \{ \overline{c} \concat \overline{c} : \overline{c} \in \beta_n \}.
 \]
 By Lemma \ref{Lem_SumDist2} on $\langle p_t : t < m \rangle$ (reordering so that these form a consecutive sequence), we get
 \[
  | \beta_{n+1} | \le | \beta_{q_0,n+1} | + 3^n \cdot 2 \cdot |D| + | \beta_n |.
 \]
 Now $| \beta_{q_0,n+1} | \le | \beta_n | \le K_n \cdot |B|$ by induction.  Hence,
 \[
  | \beta_{n+1} | \le 2 \cdot K_n \cdot |B| + 3^n \cdot 2 \cdot N_1 \cdot |\Psi| \cdot K_n \cdot |B|.
 \]
 Thereby, $| \beta_{n+1} | \le K_{n+1} \cdot |B|$.
\end{proof}

Without any further work, we obtain the fact that $| \beta_{\gamma,b} | \le K_{N_0} |B|$, hence
\[
 | \beta | \le K_{N_0} \cdot |\Gamma_0| \cdot |B|^2,
\]
showing indeed that $| \beta | = \BigO(|B|^2)$.  \emph{A priori}, this gives us the bound
\[
 | S_\varphi(B) | = \BigO(|B|^3).
\]
With another application of Lemma \ref{Lem_SumDist2}, we can get the desired result.

\begin{lem}\label{Lem_Conclusion}
 We have that $| S_\varphi(B) | \le K|B|^2$.
\end{lem}

\begin{proof}
 As before, for each $\psi \in \Psi$ and $\overline{c} \in \beta$, let $D_{\psi, \overline{c}} \subseteq \Mon_u$ with $|D_{\psi, \overline{c}}| \le N_1$ be such that, $\psi(x_1; \overline{c})$ is a boolean combination of $\delta_1(x_1; d)$ for $d \in D_{\psi,\overline{c}}$.  Let $D := \bigcup \{ D_{\psi, \overline{c}} : \psi \in \Psi_{2^{N_0}}, \overline{c} \in \beta \}$.  As before,
 \[
  |D| \le N_1 \cdot | \Psi | \cdot K_{N_0} \cdot |\Gamma_0| \cdot |B|^2.
 \]
 Again let $q_0, ..., q_{m-1} \in S$ enumerate $S$ and, for each $t < m$, choose $p_t \in S_{\delta_1}(D)$ such that $p_t \vdash q_t$.

 Consider $\psi \in \Psi$ and $\overline{c} \in \beta$.  If $p_t$ and $p_{t+1}$ disagree on $\psi(x_1; \overline{c})$, we must have $\pm \delta_1(x_1; d) \in \diff(p_t, p_{t+1})$ for some $d \in D_{\psi, \overline{c}}$.  Moreover, one such disagreement yields a change of at most $N_0$ elements between $\Vir(q_t, B)$ and $\Vir(q_{t+1}, B)$, namely possibly changing the generic $\Gamma_0$-types corresponding to $b$ and $\gamma$ (when $\psi(x_1; \overline{c}) = \psi_{\gamma,n',m',\mu}(x_1, b, \overline{c}')$).  Therefore,
 \[
  | \Vir(q_{t+1}, B) \setminus \Vir(q_t, B) | \le N_0 \cdot \dist(p_t, p_{t+1}).
 \]
 Since
 \[
  S_{\Gamma_0}(B) \subseteq \bigcup_{t<m} \Vir(q_t, B),
 \]
 Lemma \ref{Lem_SumDist2} yields
 \[
  | S_{\Gamma_0}(B) | \le | \Vir(q_0, B) | + N_0 ( 2 \cdot |D| + 2 ).
 \]
 Therefore,
 \[
  | S_{\Gamma_0}(B) | \le N_0 \cdot |\Gamma_0| \cdot |B| + 1 + 2 N_0 N_1 K_{N_0} \cdot | \Psi | \cdot |\Gamma_0| \cdot |B|^2.
 \]
 Since $|S_\varphi(B)| \le S_{\Gamma_0}(B)$, we get
 \[
  | S_\varphi(B) | \le K |B|^2.
 \]
\end{proof}

As $K$ did not depend on our choice of $B$, we get $| S_\varphi(B) | = \BigO( |B|^2 )$.  As $\varphi(x;y)$ was arbitrary with $|x| = 2$, this shows that $\pi_T(2) = 2$.  This concludes the proof of Theorem \ref{Thm_VCminVCden}.

\subsection{Conclusion}

Although it may be tempting to suppose that, using induction with the above proof, we should be able to get $\pi_T(n) = n$, this does not work with the current framework.  It is vital that both $x_0$ and $x_1$ be singletons in the above argument.  The reason induction works in Section \ref{Sect_FullyVCMin} is because $\varphi$-types correspond to generic types of balls over the same set in question.  To describe a $\varphi(x; y)$-type over $B$, one needs only $|x|$ elements \emph{from} $B$.  However, in the general argument above, $\beta \subseteq B^{2^{N_0} + 1}$.  In order to bound the size of $\beta$ (e.g., in Lemma \ref{Lem_SSmall}), we need to know \emph{a priori} that $x_1$ is a singleton.  Still, it is the hope of this author that some modification of this proof will provide a positive answer to Open Question \ref{Ques_vcminVCden}.

\end{document}